\documentclass[reqno, 11pt, letterpaper]{amsart}

\usepackage{amsmath}
\usepackage{amsfonts}
\usepackage{amssymb}
\usepackage{graphicx}
\usepackage{amsthm,graphicx,color,yfonts}
\usepackage{hyperref}
\usepackage{mathabx}
\usepackage{array}

\newtheorem{theorem}{Theorem}

\newtheorem{proposition}[theorem]{Proposition}
\newtheorem{lemma}[theorem]{Lemma}

\theoremstyle{remark}
\newtheorem{remark}[theorem]{Remark}

\usepackage{wrapfig}
\usepackage{tikz}
\usetikzlibrary{arrows,calc,decorations.pathreplacing}
\definecolor{light-gray1}{gray}{0.90}
\definecolor{light-gray2}{gray}{0.80}
\definecolor{light-gray3}{gray}{0.60}

\numberwithin{equation}{section}

\numberwithin{theorem}{section}

\numberwithin{table}{section}

\numberwithin{figure}{section}

\ifx\pdfoutput\undefined
  \DeclareGraphicsExtensions{.pstex, .eps}
\else
  \ifx\pdfoutput\relax
    \DeclareGraphicsExtensions{.pstex, .eps}
  \else
    \ifnum\pdfoutput>0
      \DeclareGraphicsExtensions{.pdf}
    \else
      \DeclareGraphicsExtensions{.pstex, .eps}
    \fi
  \fi
\fi

\title[On the NLS approximation for the NLKG]{On the NLS approximation for the nonlinear Klein-Gordon equation}

\date{\today}
\author[S. Hong]{Seokchang Hong}
\address{Department of Mathematics, Chung-Ang University, Seoul 06974, Korea}
\email{seokchangh11@cau.ac.kr}

\author[Y. Hong]{Younghun Hong}
\address{Department of Mathematics, Chung-Ang University, Seoul 06974, Korea}
\email{yhhong@cau.ac.kr}

\begin{document}
\maketitle

\begin{abstract}
In this paper, developing a new approach based on Fourier analysis methods for dispersive PDEs, we establish a low regularity NLS approximation for the one-dimensional cubic Klein-Gordon equation. Our main result includes energy class solutions which are formally asymptotically in $L^2(\mathbb{R})$. A precise rate of convergence is also obtained assuming more regularity. 
\end{abstract}

\section{Introduction}

The nonlinear Schr\"odinger equation (NLS) is a universal model describing the envelope dynamics of slowly modulating small amplitude wave packets. By multi-scaling analysis, the NLS is derived from various Hamiltonian systems, including the the Korteweg-de Vries equation \cite{Schneider1,Schneider3}, the Boussinesq equation \cite{DullSchneider}, water wave problems \cite{Zakharov, HasiOno, CSS, SW1, SW2, SW3, CGS, TW, MasmoudiNakanishi, Totz, DSW, IfrimTataru, Dull}, nonlinear optics \cite{DLPSW} and discrete models \cite{SH, Schneider2, KLS, HY}. For extensive references on the universality of the NLS, we refer to the books by Schneider and Uecker \cite{SU} and by Sulem and Sulem \cite{SuSu}.

The purpose of this article is to develop a new approach to justify the NLS approximation based on Fourier analysis methods for dispersive PDEs. In this way, we aim to include a larger class of solutions compared to the previously known dynamical system approach as well as exploring a possibility of extending the interval of approximation using the conservation laws.

To make our discussion concrete, we restrict ourselves to the one-dimensional nonlinear Klein-Gordon equation (NLKG), that is, the simplest expository textbook example \cite[Chapter 11]{SU}, given by 
\begin{equation}\label{NLKG}
\partial_t^2u-\partial_x^2u+u+u^3=0,
\end{equation}
where $u=u(t,x):\mathbb{R}\times\mathbb{R}\to\mathbb{R}$. For a small parameter $\epsilon>0$, we seek a solution of the form 
\begin{equation}\label{u ansatz}
u_\epsilon(t,x)=\epsilon A_\epsilon(\epsilon^2t,\epsilon (x-c_gt))e^{i(kx-\omega t)}+c.c.,
\end{equation}
where $A_\epsilon=A_\epsilon(t,x): \mathbb{R}\times\mathbb{R}\to\mathbb{C}$ is the amplitude function and \textit{c.c.} stands for the complex conjugate of the former term. Here, the dispersion relation and the group velocity are chosen respectively as
$$\omega=\langle k\rangle\quad\textup{and}\quad c_g=\omega'(k)=\frac{k}{\langle k\rangle},$$
where $\langle\cdot\rangle=\sqrt{1+\cdot^2}$ denotes the Japanese bracket, to cancel out $O(\epsilon)$- and $O(\epsilon^2)$-terms. Then, from the next $O(\epsilon^3)$-order terms, the cubic NLS  
\begin{equation}\label{NLS0}
2i\omega\partial_t\psi^{\textup{(NLS)}}+(1-c_g^2)\partial_x^2\psi^{\textup{(NLS)}}-3 |\psi^{\textup{(NLS)}}|^2\psi^{\textup{(NLS)}}=0,
\end{equation}
where $\psi^{\textup{(NLS)}}=\psi^{\textup{(NLS)}}(t,x): \mathbb{R}\times\mathbb{R}\to\mathbb{C}$, is derived so that the sum of two wave packets
$$\epsilon \psi^{\textup{(NLS)}}(\epsilon^2t,\epsilon (x-c_gt))e^{i(kx-\omega t)}+c.c.$$
approximates the NLKG flow. For a more detailed formal derivation and a rigorous proof for sufficiently regular flows, we refer to \cite[Chapter 11]{SU}.

In this paper, we include a larger class of solutions, namely the energy class, for the NLS approximation. Note that solutions to the NLKG \eqref{NLKG} preserve the energy 
$$\mathcal{E}(u)=\frac{1}{2}\left(\|\partial_x u\|_{L^2(\mathbb{R})}^2+\|\partial_t u\|_{L^2(\mathbb{R})}^2+\|u\|_{L^2(\mathbb{R})}^2\right)+\frac{1}{4}\|u\|_{L^4(\mathbb{R})}^4.$$
By the $\epsilon$-scaling in the ansatz \eqref{u ansatz}, it is natural to introduce the rescaled Sobolev space $H_\epsilon^s(\mathbb{R})$, for $s\in\mathbb{R}$ and $\epsilon\in(0,1]$, equipped with the norm
$$\|u\|_{H_\epsilon^s(\mathbb{R})}=\|\langle\epsilon D\rangle^su\|_{L^2(\mathbb{R})},$$
where $D=-i\partial_x$, and define $m(D)$ as the Fourier multiplier operator with symbol $m(\xi)$, i.e., $\widehat{m(D)f}(\xi)=m(\xi)\hat{f}(\xi)$. In particular, we call $H_\epsilon^1(\mathbb{R})$ the energy class.

Throughout the article, we assume that initial data satisfies 
\begin{equation}\tag{\textbf{\textup{H1}}}
\sup_{\epsilon\in(0,1]}\|\psi_{\epsilon,0}\|_{H_\epsilon^1(\mathbb{R})}\leq R<\infty
\end{equation}
and
 \begin{equation}\tag{\textbf{\textup{H2}}}
\lim_{\epsilon\to 0}\|\mathbf{1}_{|D|>\delta\epsilon^{-1/3}}\psi_{\epsilon,0}\|_{H_\epsilon^1(\mathbb{R})}=0\quad\textup{for any }\delta>0,
\end{equation}
where $\mathbf{1}_{|D|>N}$ is a high frequency cut-off with symbol $\mathbf{1}_{|\xi|>N}$. 

\begin{remark}\label{initial data remark}
$(i)$ Some regularity must be imposed on initial data, because the NLKG \eqref{NLKG} is well-posed in $H^s(\mathbb{R})$ only if $s\geq\frac{1}{2}$ (see \cite[Appendix D]{KLR}). In \textbf{\textup{(H1)}}, the regularity $s$ is chosen to be one, but we may say it is asymptotically zero due to the formal norm convergence $\|\cdot\|_{H_\epsilon^1(\mathbb{R})}\to \|\cdot\|_{L^2(\mathbb{R})}$ as $\epsilon\to 0$ even though $H_\epsilon^1(\mathbb{R})=H^1(\mathbb{R})$ as sets.\\
$(ii)$ The energy class solutions are considered with a hope to find potential applications of the energy conservation law, for instance, extending the interval of approximation. Certainly, there is a room to reduce the regularity further in the assumption \textbf{\textup{(H1)}}, but such low regularities will not be pursued here to avoid additional technical complications.\\
$(iii)$ \textbf{\textup{(H2)}} ensures that initial data $\epsilon\psi_0(\epsilon x)e^{ikx}+c.c.$ is tightly localized at two frequencies $\pm k$, with $|\xi\mp k|\ll \epsilon^{-\frac{2}{3}}$, and thus the two-wave packet structure \eqref{u ansatz} remains for long time.
\end{remark}

Our first main result asserts that the NLS approximation is valid under the above two assumptions.

\begin{theorem}[NLS approximation to the NLKG]\label{main theorem}
We assume \textbf{\textup{(H1)}}-\textbf{\textup{(H2)}} for $\{\psi_{\epsilon,0}\}_{\epsilon\in(0,1]}$. For $\epsilon\in (0,1]$, let $u_\epsilon(t)\in C_t(\mathbb{R}; H^1(\mathbb{R}))$ be the solution to the NLKG \eqref{NLKG} with initial data 
\begin{equation}\label{main theorem initial data}
\big(u_\epsilon(0), \partial_tu_\epsilon(0) \big)=\big(\epsilon\psi_{\epsilon,0}(\epsilon\cdot)e^{ikx}+c.c., \langle D\rangle (-i\epsilon\psi_{\epsilon,0}(\epsilon\cdot)e^{ikx}+c.c.)\big)\in H^1(\mathbb{R})\times L^2(\mathbb{R}),
\end{equation}
and let $\psi_\epsilon^{(\textup{NLS})}(t)\in C_t(\mathbb{R}; H_\epsilon^1(\mathbb{R}))$ be the solution to the NLS \eqref{NLS0} with initial data $\psi_{\epsilon,0}\in H_\epsilon^1(\mathbb{R})$. Then, there exists $T>0$, independent of $\epsilon\in(0,1]$, such that
\begin{equation}\label{main theorem NLS approximation}
\lim_{\epsilon\to0}\frac{1}{\sqrt{\epsilon}}\left\|u_\epsilon(t,x)-\left(\epsilon \psi_\epsilon^{(\textup{NLS})}(\epsilon^2 t, \epsilon (x-c_gt))e^{i(kx-\omega t)}+c.c.\right)\right\|_{C_t([-\frac{T}{\epsilon^2},\frac{T}{\epsilon^2}];H^1(\mathbb{R}))}=0.
\end{equation}
\end{theorem}

\begin{remark}
$(i)$ The convergence \eqref{main theorem NLS approximation} is well-known for higher regularity solutions. By a dynamical system approach, it is shown provided that the NLS flow $\psi^{(\textup{NLS})}(t)$ is in $C_t([-T,T]; H^s(\mathbb{R}))$ for $s\geq 5$ \cite[Theorem 11.2.6]{SU}. In \cite{MU}, the regularity is reduced to $s>1$. In our main theorem, the required regularity is reduced to one, but which is formally asymptotically zero (see Remark \ref{initial data remark} $(i)$).\\
$(ii)$ Higher-order corrections will not be discussed in this paper, because our main focus is on including rough solutions while higher-order corrections are typically valid for more regular solutions (see \cite[Theorem 11.2.6]{SU}).
\end{remark}

\begin{remark}[Optimality]\label{remark: optimality}
The assumptions \textbf{\textup{(H1)}}-\textbf{\textup{(H2)}} are optimal in the sense that the linear part of the reformulated equation can be approximated by the linear Schr\"odinger flow only under \textbf{\textup{(H1)}}-\textbf{\textup{(H2)}} (see Lemma \ref{linear flow convergence} and Remark \ref{remark: linear flow convergence}).
\end{remark}

\begin{remark}[Rate of convergence in \eqref{main theorem NLS approximation}]
$(i)$ By the $\epsilon$-scaling, both the NLKG flow $u_\epsilon(t)$ and the NLS ansatz $\epsilon \psi_\epsilon^{(\textup{NLS})}(\epsilon^2 t, \epsilon (x-c_gt))e^{i(kx-\omega t)}+c.c.$ are of $O(\sqrt{\omega})$ in $H^1(\mathbb{R})$. The main result \eqref{main theorem NLS approximation} justifies the NLS approximation with $o(\sqrt{\omega})$-difference.\\
$(ii)$ It does not seem possible to improve the $o(\sqrt{\omega})$-difference in \eqref{main theorem NLS approximation} for general solutions in Theorem \ref{main theorem}, because our proof relies on a density argument (see Remark \ref{remark: linear flow convergence}). 
\end{remark}

The next theorem provides a precise rate of convergence assuming more regularity.

\begin{theorem}[NLS approximation to the NLKG; rate of convergence]\label{main theorem'}
In Theorem \ref{main theorem}, we further assume that for some $s>0$,
\begin{equation}\tag{\textbf{\textup{H2'}}}
\sup_{\epsilon\in(0,1]}\|\psi_{\epsilon,0}\|_{H^s(\mathbb{R})}<\infty.
\end{equation}
Then, for any small $\eta>0$, we have
$$\left\|u_\epsilon(t,x)-\left(\epsilon \psi_\epsilon^{(\textup{NLS})}(\epsilon^2 t, \epsilon (x-c_gt))e^{i(kx-\omega t)}+c.c.\right)\right\|_{C_t([-\frac{T}{\epsilon^2},\frac{T}{\epsilon^2}];L^2(\mathbb{R}))}\lesssim \epsilon^{\min\{\frac{s}{3}+\frac{1}{2}, \frac{3}{2}-\eta\}}.$$
\end{theorem}

\begin{remark}
$(i)$ \textbf{\textup{(H2')}} is a stronger assumption, because \textbf{\textup{(H2')}} implies \textbf{\textup{(H2)}}.\\
$(ii)$ From the linear flow approximation, one can see that the $O(\epsilon^{\frac{s}{3}+\frac{1}{2}})$-rate of convergence is optimal (see Remark \ref{remark: linear flow convergence}).
\end{remark}

The main contribution of this paper is to introduce a new approach for the NLS approximation which we think is robust. Our approach is based on the reformulation of the problem as a system of integral equations \eqref{system eq}, or the Duhamel representation (see Section \ref{sec: reformulation}). It turns out that this integral representation has several crucial advantages in reducing regularity. First of all, we note that in a dynamical system approach \cite[Section 11.2]{SU}, it involves estimating the residual
$$\textup{Res}(u)=-\partial_t^2u+\partial_x^2u-u-u^3,\quad u(t,x)=\epsilon \psi^{(\textup{NLS})}(\epsilon^2 t, \epsilon (x-c_gt))e^{i(kx-\omega t)}+c.c.,$$
where $\psi^{(\textup{NLS})}$ is a solution to the NLS \eqref{NLS0}. Thus, proving smallness of the residual in $H^1(\mathbb{R})$ requires high Sobolev norm bounds, $\|\psi^{(\textup{NLS})}\|_{H^5(\mathbb{R})}<\infty$. However, such derivative terms do not appear in the integral equation \eqref{system eq}.

Secondly, the reformulation provides more detailed information about the limit procedure. Indeed, derivation of the integral equations \eqref{system eq} leads us to notice that the amplitude $A_\epsilon(t)$ in \eqref{u ansatz} includes very high frequency waves, and thus the amplitude must be separated into the core profile $\psi_\epsilon$ and the high frequency remainder $r_\epsilon$ (see Remark \ref{remark; high frequency piece}). Consequently, they must be measured separately in Sobolev norms with different scales.

Lastly, we mention that from the integral representation \eqref{system eq}, dispersive effects can be captured properly. By the reformulation, we see that the linear evolution in \eqref{system eq} is given by the propagator $S_\epsilon(t)=e^{-\frac{it}{\epsilon^2}(\langle 1+\epsilon D\rangle-\sqrt{2}-\frac{\epsilon}{\sqrt{2}}D)}$. Then, employing well-known Fourier analysis methods as Strichartz estimates and multilinear estimates involving Fourier restriction norms, one can deduce uniform bounds for nonlinear solutions which are useful to prove the NLS approximation. 

\begin{remark}
The NLS is derived from the NLKG in a different context, namely as a non-relativistic limit \cite{MNO}, but it does not have the technical issue caused by the high frequency remainder.
\end{remark}

\begin{remark}
An interesting question is whether the $O(\frac{1}{\epsilon^2})$-interval of approximation can be extended.  In \cite{FeolaGiuliani}, a positive answer is given for more complicated quadratic NLKG in a periodic setting but with more regular solutions. As an alternative approach, one may attempt to use the energy conservation law to extend the interval of approximation. Nevertheless, we are currently unable to do that. Indeed, our proof heavily relies on the two-wave packet structure $\epsilon\psi_{\epsilon,0}(\epsilon\cdot)e^{ikx}+c.c.$ at initial time, but the structure will be broken up as time goes, because the nonlinearity immediately generates different frequency modes. Unfortunately, the energy conservation law does not seem to control this procedure straightforwardly. Therefore, it will be left to our future work.
\end{remark}

\subsection{Organization of the paper}
In Section \ref{sec: reformulation}, we derive the system of integral equations \eqref{system eq}. In Section \ref{sec: linear flow} and \ref{sec: well-posedness}, we investigate the properties of the linear part of \eqref{system eq} and use them to prove basic well-posedness of the system. Then, in Section \ref{sec: core profile} and \ref{sec: remainder}, we prove more refined estimates for the core profile (smallness of high frequencies) and the remainder (smallness). Finally, in Section \ref{sec: proof of the main results}, we complete the proof of the main results. 

\subsection{Acknowledgement}
This work was supported by the National Research Foundation of Korea (NRF) grant funded by the Korea government (MSIT) (NRF-2020R1A2C4002615).

\section{Reformulation of the problem}\label{sec: reformulation}

To begin with, we present a reformulation of the cubic nonlinear Klein-Gordon equation \eqref{NLKG} in a way that Fourier analysis methods can be properly employed. 

\subsection{Derivation of the equation for the amplitude function}
For numerical simplicity, we fix $k=1$ by scaling, and look for a solution of the form
$$u_\epsilon(t,x)=\epsilon A_\epsilon(\epsilon^2t,\epsilon (x-\tfrac{t}{\sqrt{2}}))e^{i(x-\sqrt{2}t)}+c.c..$$
In order to find the equation for the amplitude $A_\epsilon(t,x)$, rescaling by $v_\epsilon(t,x)=\frac{1}{\epsilon}u_\epsilon(\frac{t}{\epsilon^2},\frac{x}{\epsilon})$, we deduce 
\begin{equation}\label{rescaled NLKG}
\partial_t^2v_\epsilon+\frac{1}{\epsilon^4}(1-\epsilon^2\partial_x^2)v_\epsilon-\frac{1}{\epsilon^2}v_\epsilon^3=0
\end{equation}
with initial data $(v_\epsilon(0), \partial_tv_\epsilon(0)\big)=\big(2\textup{Re}(e^{\frac{ix}{\epsilon}}\psi_{\epsilon,0}), \tfrac{2}{\epsilon^2}\langle\epsilon\partial_x\rangle\textup{Im}(e^{\frac{ix}{\epsilon}}\psi_{\epsilon,0}))$. We note that the equation \eqref{rescaled NLKG} in a strong form is given by 
$$\begin{aligned}
v_\epsilon(t)&=\left\{e^{-\frac{it}{\epsilon^2}\langle\epsilon D\rangle}e^{\frac{ix}{\epsilon}}\psi_{\epsilon,0}-\frac{i}{2}\int_0^t e^{-\frac{i(t-t_1)}{\epsilon^2}\langle\epsilon D\rangle}\frac{1}{\langle\epsilon D\rangle} v_\epsilon(t_1)^3dt_1\right\}+c.c.
\end{aligned}$$
with
$$v_\epsilon(t,x)=e^{-\frac{itD}{\epsilon\sqrt{2}}}e^{\frac{i}{\epsilon}(x-\frac{t}{\epsilon\sqrt{2}})}A_\epsilon(t,x)+c.c..$$
Next, complexifying the equation, we impose that the amplitude function $A_\epsilon(t)$ obeys the equation
$$\begin{aligned}
A_\epsilon(t)&=e^{-\frac{i}{\epsilon}(x-\frac{t}{\epsilon\sqrt{2}})} e^{-\frac{it}{\epsilon^2}(\langle\epsilon D\rangle-\frac{\epsilon D}{\sqrt{2}})}e^{\frac{ix}{\epsilon}}\psi_{\epsilon,0}\\
&\quad-\frac{i}{2}\int_0^t e^{-\frac{i}{\epsilon}(x-\frac{t}{\epsilon\sqrt{2}})} e^{-\frac{i(t-t_1)}{\epsilon^2}(\langle\epsilon D\rangle-\frac{\epsilon D}{\sqrt{2}})}\frac{1}{\langle\epsilon D\rangle}\left(e^{\frac{i}{\epsilon}(x-\frac{t_1}{\epsilon\sqrt{2}})}A_\epsilon(t_1)+c.c.\right)^3 dt_1.
\end{aligned}$$
We observe that $e^{-\frac{i}{\epsilon}(x-\frac{t}{\epsilon\sqrt{2}})} e^{-\frac{it}{\epsilon^2}(\langle\epsilon D\rangle-\frac{\epsilon D}{\sqrt{2}})}e^{\frac{ix}{\epsilon}}$ is the Fourier multiplier of symbol $e^{-itp_\epsilon(\xi)}$, where
\begin{equation}\label{varphi}
p_\epsilon(\xi):=\frac{1}{\epsilon^2}\left(\langle 1+\epsilon\xi\rangle-\sqrt{2}-\frac{\epsilon}{\sqrt{2}}\xi\right),
\end{equation}
and similarly, $e^{-\frac{i}{\epsilon}(x-\frac{t}{\epsilon\sqrt{2}})} e^{-\frac{i(t-t_1)}{\epsilon^2}(\langle\epsilon D\rangle-\frac{\epsilon D}{\sqrt{2}})}\frac{1}{\langle\epsilon D\rangle}e^{\frac{i}{\epsilon}(x-\frac{t_1}{\epsilon\sqrt{2}})}$ has symbol $\frac{e^{-i(t-t_1)p_\epsilon(\xi)}}{\langle 1+\epsilon\xi\rangle}$. Thus, introducing the linear propagator 
$$S_\epsilon(t)=e^{-itp_\epsilon(D)},$$
the equation can be written as 
$$\begin{aligned}
A_\epsilon(t)&=S_\epsilon(t)\psi_{\epsilon,0}-\frac{i}{2}\int_0^tS_\epsilon(t-t_1)\frac{1}{\langle 1+\epsilon D\rangle}e^{-\frac{i}{\epsilon}(x-\frac{t_1}{\epsilon\sqrt{2}})}\left(A_\epsilon(t_1)e^{\frac{i}{\epsilon}(x-\frac{t_1}{\epsilon\sqrt{2}})}+c.c.\right)^3dt_1.
\end{aligned}$$
Finally, expanding the nonlinear term as
$$e^{-\frac{i}{\epsilon}(x-\frac{t}{\epsilon\sqrt{2}})}\left(A_\epsilon e^{\frac{i}{\epsilon}(x-\frac{t}{\epsilon\sqrt{2}})}+c.c.\right)^3=3\big(|A_\epsilon|^2A_\epsilon+\mathcal{R}_\epsilon(A_\epsilon)\big),$$
where
$$\mathcal{R}_\epsilon(A_\epsilon):=\frac{1}{3}e^{\frac{2i}{\epsilon}(x-\frac{t}{\epsilon\sqrt{2}})}A_\epsilon^3+e^{-\frac{2i}{\epsilon}(x-\frac{t}{\epsilon\sqrt{2}})}|A_\epsilon|^2\bar{A}_\epsilon+\frac{1}{3}e^{-\frac{4i}{\epsilon}(x-\frac{t}{\epsilon\sqrt{2}})}\bar{A}_\epsilon^3,$$
we obtain the amplitude equation in a compact form 
\begin{equation}\label{NLKG'}
A_\epsilon(t)=S_\epsilon(t)\psi_{\epsilon,0}-\frac{3i}{2}\int_0^t S_\epsilon(t-t_1)\frac{1}{\langle1+\epsilon D\rangle}\left(|A_\epsilon|^2A_\epsilon+\mathcal{R}_\epsilon(A_\epsilon)\right)(t_1)dt_1.
\end{equation}

\subsection{Decomposition of the amplitude equation}

It turns out, however, that the equation \eqref{NLKG'} by itself is rather difficult to analyze, because on the frequency side, the solution consists of a large localized part and small very high frequency pieces. 

\begin{remark}\label{remark; high frequency piece}
The solution $A_\epsilon$ to \eqref{NLKG'} is never localized in frequencies of order one. Indeed, even if $\hat{A}_\epsilon(t,\xi)$ is mostly localized near the origin at time $t=0$, the nonlinear term 
$$-\frac{i}{2}\int_0^t S_\epsilon(t-t_1)\frac{1}{\langle1+\epsilon D\rangle}\left(e^{\frac{2i}{\epsilon}(x-\frac{t_1}{\epsilon\sqrt{2}})}A_\epsilon^3+3e^{-\frac{2i}{\epsilon}(x-\frac{t_1}{\epsilon\sqrt{2}})}|A_\epsilon|^2\bar{A}_\epsilon+e^{-\frac{4i}{\epsilon}(x-\frac{t_1}{\epsilon\sqrt{2}})}\bar{A}_\epsilon^3\right)(t_1)dt_1$$
immediately generates very high frequencies $\xi=\mp\frac{2}{\epsilon}, -\frac{4}{\epsilon}$. One may expect that these high frequency pieces vanish as $\epsilon\to 0$ due to fast dispersion. However, they might be measured largely in a standard Sobolev norm $\|\cdot\|_{H^s(\mathbb{R})}$. On the other hand, if one employs the rescaled Sobolev norm $\|\cdot\|_{H_\epsilon^s(\mathbb{R})}$ for \eqref{NLKG'}, one becomes unable to capture how much portion of $A_\epsilon$ is localized in frequencies. 
\end{remark}

A key idea to resolve the problem is to separate the amplitude function into the core profile and the high frequency remainder,
$$A_\epsilon(t,x)=\psi_\epsilon(t,x)+r_\epsilon(t,x),$$
and impose that $(\psi_\epsilon(t),r_\epsilon(t))$ solves the system of equations 
\begin{equation}\label{system eq}
\left\{\begin{aligned}
\psi_\epsilon(t)&=S_\epsilon(t)\psi_{\epsilon,0}-\frac{3i}{2}\int_0^t S_\epsilon(t-t_1)\frac{1}{\langle 1+\epsilon D\rangle}\left(|\psi_\epsilon|^2\psi_\epsilon\right)(t_1)dt_1,\\
r_\epsilon(t)&=-\frac{3i}{2}\int_0^t S_\epsilon(t-t_1)\frac{1}{\langle 1+\epsilon D\rangle}\big(|A_\epsilon|^2(A_\epsilon)-|\psi_\epsilon|^2\psi_\epsilon+\mathcal{R}_\epsilon(A_\epsilon)\big)(t_1)dt_1.
\end{aligned}\right.
\end{equation}
In a sequel, they will be estimated separately.

\section{Linear rescaled Klein-Gordon flow}\label{sec: linear flow}

In this section, we investigate properties of the linear rescaled Klein-Gordon flow 
$$S_\epsilon(t)=e^{-itp_\epsilon(D)}=e^{-\frac{it}{\epsilon^2}(\langle 1+\epsilon D\rangle-\sqrt{2}-\frac{\epsilon}{\sqrt{2}}D)}$$
focusing on its connection to the linear Schr\"odinger flow $e^{\frac{it}{4\sqrt{2}}\partial_x^2}$.

\subsection{Convergence of the linear flow}
By the formal Taylor series expansion
$$\langle 1+\epsilon D\rangle=\sqrt{2}+\frac{\epsilon}{\sqrt{2}}D+\frac{\epsilon^2}{4\sqrt{2}}D^2-\frac{\epsilon^3}{8\sqrt{2}}D^3+\cdots,$$
it is expected that $S_\epsilon(t)$ converges to $e^{\frac{it}{4\sqrt{2}}\partial_x^2}$ as $\epsilon\to 0$. It can be stated rigorously as follows.

\begin{lemma}[Convergence of the linear flow $S_\epsilon(t)$]\label{linear flow convergence}
Let $\epsilon\in(0,1]$. 
$(i)$ For any sufficiently small $\delta>0$, we have
$$\|S_\epsilon(t)u_0-e^{\frac{it}{4\sqrt{2}}\partial_x^2}u_0\|_{L^2(\mathbb{R})}\lesssim \delta^3\|u_0\|_{L^2(\mathbb{R})}+\|P_{>\delta\epsilon^{-1/3}}u_0\|_{L^2(\mathbb{R})}.$$
$(ii)$ If $u_0\in H^s(\mathbb{R})$, then 
$$\|S_\epsilon(t)u_0-e^{\frac{it}{4\sqrt{2}}\partial_x^2}u_0\|_{L^2(\mathbb{R})}\lesssim \epsilon^{\frac{s}{3}}\|u_0\|_{H^s(\mathbb{R})}.$$
\end{lemma}

\begin{proof}
By the Plancherel theorem, we have
$$\|S_\epsilon(t)u_0-e^{\frac{it}{4\sqrt{2}}\partial_x^2}u_0\|_{L^2}\sim \|(e^{-it(p_\epsilon(\xi)-\frac{1}{4\sqrt{2}}\xi^2)}-1)\hat{u}_0\|_{L^2}.$$
Note that
\begin{equation}\label{symbol difference}
|e^{-it(p_\epsilon(\xi)-\frac{1}{4\sqrt{2}}\xi^2)}-1|\lesssim\min\left\{\epsilon|\xi|^3,1\right\},
\end{equation}
because $|p_\epsilon(\xi)-\frac{1}{4\sqrt{2}}\xi^2|\lesssim \epsilon|\xi|^3$ by Taylor's theorem. Thus, the lemma follows. 
\end{proof}

\begin{remark}\label{remark: linear flow convergence}
From the proof of Lemma \ref{linear flow convergence} (see \eqref{symbol difference} in particular), one can see that \textup{\textbf{(H2)}} is the minimal requirement for the convergence $\|S_\epsilon(t)u_0-e^{\frac{it}{4\sqrt{2}}\partial_x^2}u_0\|_{H_\epsilon^1(\mathbb{R})}\to 0$ as well as \textup{\textbf{(H2')}} is optimal for the $O(\epsilon^{\frac{s}{3}})$-rate of convergence. 
\end{remark}

\subsection{Uniform linear Strichartz estimates}
For $\epsilon\in(0,1]$, we define the rescaled Sobolev norms by 
$$\|u\|_{W_\epsilon^{s,r}(\mathbb{R})}:=\|\langle \epsilon D\rangle^s u\|_{L^r(\mathbb{R})}\quad\textup{and}\quad \|u\|_{H_\epsilon^s(\mathbb{R})}:=\|u\|_{W_\epsilon^{s,2}(\mathbb{R})}.$$
Using a smooth cut-off $\chi\in C_c^\infty(\mathbb{R})$ such that $\textup{supp}\chi\subset \{\xi: \frac{1}{2}<|\xi|<4\}$, $\chi\equiv 1$ on $1\leq |\xi|\leq 2$ and $\sum_{N\in 2^{\mathbb{Z}}}\chi(\frac{\cdot}{N})\equiv 1$, we define the Littlewood-Paley projection $P_N$ by
$$\widehat{P_N f}(\xi)=\chi(\tfrac{\xi}{N})\hat{f}(\xi).$$

Our analysis heavily relies on dispersive properties of the linear Klein-Gordon flow $S_\epsilon(t)$ in the form of Strichartz estimates. Indeed, such estimates are well-known for the standard Klein-Gordon equation, and in essence, the same holds for the rescaled one. However, in consideration of the $\epsilon\to 0$ limit, possible $\epsilon$-dependences in the inequalities must be clarified.

\begin{lemma}[Uniform linear Strichartz estimates]\label{Strichartz estimates}
Let $\epsilon\in(0,1]$. Suppose that $q,r,\tilde{q},\tilde{r}\geq 2$ and
$$\frac{2}{q}+\frac{1}{r}=\frac{2}{\tilde{q}}+\frac{1}{\tilde{r}}=\frac{1}{2}.$$
Then, we have
$$\|S_\epsilon(t)u_0\|_{L_t^q(\mathbb{R}; L^r(\mathbb{R}))}\lesssim \|u_0\|_{H_\epsilon^{\frac{3}{q}}(\mathbb{R})}$$
and
$$\left\|\int_0^t S_\epsilon(t-t_1)F(t_1)dt_1\right\|_{L_t^q(\mathbb{R}; L^r(\mathbb{R}))}\lesssim \|F\|_{L_t^{\tilde{q}'}(\mathbb{R}; W_{\epsilon}^{3(\frac{1}{q}+\frac{1}{\tilde{q}}), \tilde{r}'}(\mathbb{R}))}.$$
\end{lemma}

\begin{proof}
Since $p_\epsilon''(\epsilon\xi)=\frac{1}{\langle 1+\epsilon\xi\rangle^3}$, the van der Corput lemma deduces that 
$$\|S_\epsilon(t)P_Nu_0\|_{L_x^\infty}=\left\|\int_{\mathbb{R}}\left\{\int_{\mathbb{R}}e^{-itp_\epsilon(\xi)}e^{i(x-y)\xi} \chi(\tfrac{\xi}{N})dy\right\}u_0(y)dy\right\|_{L_x^\infty}\lesssim \frac{1}{\langle\epsilon N\rangle^{3/2}|t|^{1/2}}\|u_0\|_{L^1}.$$
Thus, the standard interpolation argument (see Keel-Tao \cite{KT} for instance), together with the Littlewood-Paley inequality, yields the desired inequalities.
\end{proof}

\begin{remark}
Strichartz estimates for the rescaled flow $S_\epsilon(t)$ require some regularities like the standard linear Klein-Gordon flow $e^{it\langle D\rangle}$. On the other hand, in the $\epsilon\to 0$ limit, they are closely related to the Strichartz estimates for the linear Schr\"odinger flow $e^{\frac{it}{4\sqrt{2}}\partial_x^2}$, i.e., 
$$\|e^{\frac{it}{4\sqrt{2}}\partial_x^2}u_0\|_{L_t^q(\mathbb{R}; L^r(\mathbb{R}))}\lesssim \|u_0\|_{L^2(\mathbb{R})},$$
$$\left\|\int_0^t e^{\frac{i(t-t_1)}{4\sqrt{2}}\partial_x^2}F(t_1)dt_1\right\|_{L_t^q(\mathbb{R}; L^r(\mathbb{R}))}\lesssim \|F\|_{L_t^{\tilde{q}'}(\mathbb{R}; L^{\tilde{r}'}(\mathbb{R}))}$$
as the rescaled Sobolev norm $\|\cdot\|_{W_\epsilon^{s,r}(\mathbb{R})}$ formally converges to the $L^{r}(\mathbb{R})$-norm.
\end{remark}

Next, we introduce two types of Fourier restriction norm spaces, namely the \textit{Bourgain} spaces \cite{B}, associated with the linear propagator $S_\epsilon(t)$. We define $X^{s,b}$ (resp., $X_\epsilon^{s,b}$) as the completion of Schwartz functions with respect to 
$$\|u\|_{X^{s,b}} := \|\langle\xi\rangle^s\langle\tau+p_\epsilon(\xi)\rangle^b \tilde{u}(\tau,\xi)\|_{L^2_{\tau,\xi}(\mathbb{R}\times\mathbb{R})}$$
$$\left(\textup{resp., }\|u\|_{X^{s,b}} := \|\langle\epsilon\xi\rangle^s\langle\tau+p_\epsilon(\xi)\rangle^b \tilde{u}(\tau,\xi)\|_{L^2_{\tau,\xi}(\mathbb{R}\times\mathbb{R})}\right),$$
where $\tilde{u}(\tau,\xi)$ denotes the space-time Fourier transform of $u(t,x)$. Then, given $T>0$, we define $X^{s,b}([-T,T])$ (resp., $X_\epsilon^{s,b}([-T,T])$) with the norm
$$\|u\|_{X^{s,b}([-T,T])}:=\inf \Big\{ \|v\|_{X^{s,b}}: v=u\textup{ on }[-T,T]\Big\}$$
$$\left(\textup{resp., }\|u\|_{X_\epsilon^{s,b}([-T,T])}:=\inf\Big\{ \|v\|_{X_\epsilon^{s,b}}: v=u\textup{ on }[-T,T]\Big\}\right).$$
Later, we will frequently use the following basic mapping properties (refer to \cite{B, Tao book}) and the trilinear estimate involving the Fourier restriction spaces. 
\begin{lemma}
$(i)$ (Inhomogeneous estimate) For any $b>\frac{1}{2}$, 
$$\left\|\int_0^t S_\epsilon(t-t_1)F(t_1)dt_1\right\|_{X^{0,b}([-T,T])}\lesssim \|F\|_{X^{0,-(1-b)}([-T,T])}.$$
$(ii)$ (Transference principle) If $b>\frac12$, $\frac2q+\frac1r=\frac12$ and $q,r\ge2$, then
$$\|u\|_{L^q_t([-T,T];L^r_x(\mathbb{R}))} \lesssim \|u\|_{X^{\frac{3}{q},b}_\epsilon([-T,T])}.$$ 	
\end{lemma}

\begin{lemma}[Trilinear estimate]\label{basic trilinear estimate}
For $b\in(\frac{1}{2},1]$, we have
$$\begin{aligned}
\|v_1v_2v_3\|_{X^{0,-(1-b)}([-T,T])}&\lesssim \|v_1v_2v_3\|_{L_t^{2b}(\mathbb{R};L_x^2(\mathbb{R})}\\
&\lesssim T^{\frac{1-b}{2b}}\|v_1\|_{X_\epsilon^{1,b}([-T,T])}\|v_2\|_{X_\epsilon^{1,b}([-T,T])}\|v_3\|_{X^{0,b}([-T,T])}.
\end{aligned}$$
\end{lemma}

\begin{proof}
Interpolating the two basic inequalities $\|u\|_{L_t^\infty L_x^2}\lesssim \|u\|_{X_\epsilon^{0,b}}$ and $\|u\|_{L_t^2 L_x^2}\lesssim \|u\|_{X_\epsilon^{0,0}}$, we obtain 
$\|u\|_{L_t^{\frac{2b}{2b-1}}L_x^2}\lesssim\|u\|_{X_\epsilon^{0,1-b}}$ whose dual inequality is given by $\|u\|_{X_\epsilon^{0,-(1-b)}}\lesssim \|u\|_{L_t^{2b}L_x^2}$. Thus, by H\"older's inequality and Strichartz estimates, we prove that 
$$\begin{aligned}
\|v_1v_2v_3\|_{X^{0,-(1-b)}}&\lesssim\|v_1v_2v_3\|_{L_t^{2b}L_x^2}\lesssim T^{\frac{1-b}{2b}}\|v_1\|_{L_t^4L_x^\infty}\|v_2\|_{L_t^4L_x^\infty}\|v_3\|_{C_t L_x^2}\\
&\lesssim T^{\frac{1-b}{2b}}\|v_1\|_{X_\epsilon^{1,b}}\|v_2\|_{X_\epsilon^{1,b}}\|v_3\|_{X^{0,b}}.
\end{aligned}$$
\end{proof}

\section{Well-posedness of the nonlinear Klein-Gordon equation}\label{sec: well-posedness}

We present basic well-posedness results for the rescaled NLKG \eqref{rescaled NLKG} and the system \eqref{system eq}. Consequently, we confirm that $\psi_\epsilon(t,x)+r_\epsilon(t,x)+c.c.$, from \eqref{system eq}, is a unique solution to the NLKG \eqref{rescaled NLKG} with initial data $\big(2\textup{Re}(e^{\frac{ix}{\epsilon}}\psi_{\epsilon,0}), \tfrac{2}{\epsilon^2}\langle\epsilon\partial_x\rangle\textup{Im}(e^{\frac{ix}{\epsilon}}\psi_{\epsilon,0})\big)$ (see Remark \ref{NLKG system relation} below).

The well-posedness of the standard NLKG \eqref{NLKG} is well-known. For the rescaled one, it is rephrased as follows.

\begin{proposition}[Global well-posedness for the rescaled NLKG]\label{GWP for rescaled NLKG}
Let $\epsilon\in(0,1]$. Given initial data $(v_{\epsilon}(0), \partial_tv_{\epsilon}(0))\in H_\epsilon^1(\mathbb{R})\times L^2(\mathbb{R})$, there exists a unique global strong solution $v_\epsilon(t)\in C_t(\mathbb{R}; H_\epsilon^1)$ to the rescaled NLKG \eqref{rescaled NLKG}. Moreover, $v_\epsilon(t)$ preserves the energy
$$\mathcal{E}_\epsilon[v]=\frac{\epsilon^2}{2}\|\partial_x v\|_{L^2(\mathbb{R})}^2+\frac{1}{2}\|v\|_{L^2(\mathbb{R})}^2+\frac{\epsilon^4}{2}\|\partial_tv\|_{L^2(\mathbb{R})}^2+\frac{\epsilon^2}{4}\|v\|_{L^4(\mathbb{R})}^4.$$
\end{proposition}

\begin{proof}[Sketch of Proof]
The proof follows from a standard contraction mapping argument. Indeed, by the Sobolev embedding $H^1(\mathbb{R})\hookrightarrow L^\infty(\mathbb{R})$, one can show that the corresponding integral equation is locally well-posed in $C_t([-T_\epsilon,T_\epsilon]; H_\epsilon^1)$ for sufficiently small $T_\epsilon>0$ possibly depending on $\epsilon\in(0,1]$. Then, it can be upgraded to global well-posedness, because the conservation law prevents the solution from blowing up in finite time.
\end{proof}

\begin{remark}
As for well-posedness of the rescaled NLKG \eqref{rescaled NLKG}, it is not necessary to employ Strichartz estimates in the previous section. Nevertheless, if one uses $\epsilon$-depending inequalities, e.g., the Sobolev inequality $\|v\|_{L^\infty(\mathbb{R})}\lesssim \frac{1}{\epsilon}\|v\|_{H_\epsilon^1(\mathbb{R})}$, the size of the interval $[-T_\epsilon, T_\epsilon]$ in the contraction mapping argument would shrink to $\{0\}$ as $\epsilon\to 0$.
\end{remark}

Next, using Strichartz estimates, we establish well-posedness for the system \eqref{system eq}.

\begin{proposition}[Local well-posedness for the system \eqref{system eq}]\label{LWP for system}
Let $\epsilon\in(0,1]$ and $b\in(\frac{1}{2},1)$. We assume that $\{\psi_{\epsilon,0}\}_{\epsilon\in(0,1]}$ satisfies \textup{\textbf{(H1)}} for some $R>0$. Then, there exists $0<T\lesssim R^{-\frac{4b}{1-b}}$, independent of $\epsilon\in(0,1]$, such that the system \eqref{system eq} has a unique solution $(\psi_\epsilon(t), r_\epsilon(t))\in C_t([-T,T]; H_\epsilon^1(\mathbb{R})\times H_\epsilon^1(\mathbb{R}))$ with 
\begin{equation}\label{core bound}
\|\psi_\epsilon\|_{X_\epsilon^{1,b}([-T,T])}\lesssim R\quad\textup{and}\quad\|r_\epsilon\|_{X_\epsilon^{1,b}([-T,T])}\lesssim T^{\frac{1-b}{2b}}R^3.
\end{equation}
Moreover, we have
\begin{equation}\label{high Sobolev norm core bound}
\|\psi_\epsilon\|_{X^{s,b}([-T,T])}\lesssim \|\psi_{\epsilon,0}\|_{H^s(\mathbb{R})}.
\end{equation}
\end{proposition}

\begin{proof}
We define
$$\Phi_1(\psi)(t):=S_\epsilon(t)\psi_{\epsilon,0}-\frac{3i}{2}\int_0^t S_\epsilon(t-t_1)\frac{1}{\langle 1+\epsilon D\rangle}|\psi|^2\psi(t_1)dt_1.$$
Then, by Strichartz estimates and the trilinear estimate (Lemma \ref{basic trilinear estimate}), one can show that 
$$\begin{aligned}
\|\Phi_1(\psi)\|_{X_\epsilon^{1,b}}&\leq cR+cT^{\frac{1-b}{2b}}\|\psi\|_{X_\epsilon^{1,b}}^3\leq cR+cT^{\frac{1-b}{2b}}(2cR)^3,\\
\|\Phi_1(\psi)-\Phi_1(\psi')\|_{X_\epsilon^{1,b}}&\leq c T^{\frac{1-b}{2b}}\big(\|\psi\|_{X_\epsilon^{1,b}}^2+\|\psi'\|_{X_\epsilon^{1,b}}^2\big)\|\psi-\psi'\|_{X_\epsilon^{1,b}}\\
&\leq 2c T^{\frac{1-b}{2b}}(2cR)^2\|\psi-\psi'\|_{X_\epsilon^{1,b}}\leq\frac{1}{2}\|\psi-\psi'\|_{X_\epsilon^{1,b}},
\end{aligned}$$
provided that $\|\psi\|_{X_\epsilon^{1,b}}, \|\psi'\|_{X_\epsilon^{1,b}}\leq 2cR$. We take small $T>0$ such that $c T^{\frac{1-b}{2b}}(2cR)^2\leq\frac{1}{4}$. Then, it follows that $\Phi_1$ is contractive on $\{\psi\in X_\epsilon^{1,b}: \|\psi\|_{X_\epsilon^{1,b}}\leq 2cR\}$, and it has a unique fixed point $\psi_\epsilon$ with $\|\psi_\epsilon\|_{X_\epsilon^{1,b}}\leq 2cR$.

Next, we consider
$$\begin{aligned}
\Phi_2(r)(t):&=-\frac{3i}{2}\int_0^t S_\epsilon(t)\frac{1}{\langle 1+\epsilon D\rangle}\left(|\psi_\epsilon+r|^2(\psi_\epsilon+r)-|\psi_\epsilon|^2\psi_\epsilon+\mathcal{R}_\epsilon(\psi_\epsilon+r)\right)(t_1)dt_1.
\end{aligned}$$
We again use Strichartz estimates and the trilinear estimate (Lemma \ref{basic trilinear estimate}) to obtain 
$$\begin{aligned}
\|\Phi_2(r)\|_{X_\epsilon^{1,b}}&\leq c T^{\frac{1-b}{2b}}\big(\|\psi_\epsilon\|_{X_\epsilon^{1,b}}^3+\|r\|_{X_\epsilon^{1,b}}^3\big)\leq 2c T^{\frac{1-b}{2b}}(2cR)^3,\\
\|\Phi_2(r)-\Phi_2(r')\|_{X_\epsilon^{1,b}}&\leq c T^{\frac{1-b}{2b}}\big(\|\psi_\epsilon\|_{X_\epsilon^{1,b}}^2+\|r\|_{X_\epsilon^{1,b}}^2+\|r'\|_{X_\epsilon^{1,b}}^2\big)\|r-r'\|_{X_\epsilon^{1,b}}\\
&\leq 3c T^{\frac{1-b}{2b}}(2cR)^2\|r-r'\|_{X_\epsilon^{1,b}}\leq\frac{3}{4}\|r-r'\|_{X_\epsilon^{1,b}}
\end{aligned}$$
if $\|r\|_{X_\epsilon^{1,b}}, \|r'\|_{X_\epsilon^{1,b}}\leq 2c T^{\frac{1-b}{2b}}(2cR)^3\leq cR$. Therefore, $r=\Phi_2(r)$ has a unique solution $r_\epsilon$ with $\|r_\epsilon\|_{X_\epsilon^{1,b}}\leq 2c T^{\frac{1-b}{2b}}(2cR)^3$. 

It remains to show \eqref{high Sobolev norm core bound}. Indeed, by Strichartz estimates, we have
 $$\|\psi_\epsilon\|_{X^{s,b}}\lesssim \|\psi_{\epsilon,0}\|_{H^s}+\||\psi_\epsilon|^2\psi_\epsilon\|_{X^{s,-(1-b)}}.$$
Slightly modifying the proof of Lemma \ref{basic trilinear estimate}, one can show that 
\begin{equation}\label{high Sobolev norm core nonlinear bound}
\begin{aligned}
\||\psi_\epsilon|^2\psi_\epsilon\|_{X^{s,-(1-b)}}&\lesssim \||\psi_\epsilon|^2\psi_\epsilon\|_{L_t^{2b}H_x^{s}}\lesssim T^{\frac{1-b}{2b}}\|\psi_\epsilon\|_{L_t^4L_x^\infty}^2\|\psi_\epsilon\|_{C_tH_x^{s}}\\
&\lesssim T^{\frac{1-b}{2b}}\|\psi_\epsilon\|_{X_\epsilon^{1,b}}^2\|\psi_\epsilon\|_{X^{s,b}}\lesssim T^{\frac{1-b}{2b}}R^2\|\psi_\epsilon\|_{X^{s,b}}.
\end{aligned}
\end{equation}
Thus, taking smaller $T\ll R^{-\frac{4b}{1-b}}$ if necessary, we obtain $\|\psi_\epsilon\|_{X^{s,b}}\lesssim \|\psi_{\epsilon,0}\|_{H^s}$.
\end{proof}

\begin{remark}\label{NLKG system relation}
From the construction, summing two equations in the system \eqref{system eq} and their c.c.'s,  one can see that $\psi_\epsilon(t,x)+r_\epsilon(t,x)+c.c.$ is a strong solution to the rescaled NLKG \eqref{rescaled NLKG} and that it is included in $C_t([-T,T];H_\epsilon^1(\mathbb{R}))$ by the transference principle. Therefore, by uniqueness, the solution constructed from the system \eqref{system eq} is a unique strong solution to the rescaled NLKG  with initial data $(2\textup{Re}(e^{\frac{ix}{\epsilon}}\psi_{\epsilon,0}), \frac{2}{\epsilon^2}\langle\epsilon\partial_x\rangle\textup{Im}(e^{\frac{ix}{\epsilon}}\psi_{\epsilon,0}))$.
\end{remark}

\begin{remark}
We have an analogous well-posedness result for the NLS
\begin{equation}\label{NLS}
i\partial_t\psi^{\textup{(NLS)}}+\frac{1}{4\sqrt{2}}\partial_x^2\psi^{\textup{(NLS)}}-\frac{3}{2\sqrt{2}}|\psi^{\textup{(NLS)}}|^2\psi^{\textup{(NLS)}}=0.
\end{equation}
Indeed, by the standard argument (see Cazenave \cite{Cazenave}) but with the Strichartz estimate
$$\|e^{\frac{it}{4\sqrt{2}}\partial_x^2}u_0\|_{C_t(\mathbb{R}; L^2(\mathbb{R}))\cap L_t^4(\mathbb{R}; L_x^\infty(\mathbb{R}))}\lesssim\|u_0\|_{L^2(\mathbb{R})},$$
one can show that there exists $T>0$, independent of $\epsilon\in(0,1]$, such that the following hold.\\
$(i)$ If \textup{\textbf{(H1)}} holds, the solution $\psi_\epsilon^{\textup{(NLS)}}\in C_t([-T,T];H_\epsilon^1(\mathbb{R}))$ to the NLS \eqref{NLS} with initial data $\psi_{\epsilon,0}$ exists, and
\begin{equation}\label{NLS property 1}
\|\psi_\epsilon^{\textup{(NLS)}}\|_{C_t([-T,T]; H_\epsilon^1(\mathbb{R}))\cap L_t^4([-T,T]; L^\infty(\mathbb{R}))}\lesssim \|\psi_{\epsilon,0}\|_{H_\epsilon^1(\mathbb{R})}.
\end{equation}
$(ii)$ If we further assume \textup{\textbf{(H2')}} for some $s>0$, then
\begin{equation}\label{NLS property 2}
\|\psi_\epsilon^{\textup{(NLS)}}\|_{C_t([-T,T]; H^s(\mathbb{R}))}+\big\||\psi_\epsilon^{\textup{(NLS)}}|^2\psi_\epsilon^{\textup{(NLS)}}\big\|_{L_t^1([-T,T]; H^s(\mathbb{R}))}\lesssim \|\psi_{\epsilon,0}\|_{H^s(\mathbb{R})}.
\end{equation}

\end{remark}

\section{High frequency estimate for the core profile}\label{sec: core profile}

We recall that the convergence of the linear flow $S_\epsilon(t)$ requires smallness of high frequencies (see Lemma \ref{linear flow convergence} and Remark \ref{remark: linear flow convergence}). Thus, one may expect that a similar estimate would be needed for the nonlinear problem. In this section, we prove the following proposition for the core profile. 

\begin{proposition}[High frequency estimates for the core profile]\label{core high frequency}
Let $\epsilon\in(0,1]$ and $b\in(\frac{1}{2},1)$. We assume \textbf{\textup{(H1)}} and \textbf{\textup{(H2)}} for $\{\psi_{\epsilon,0}\}_{\epsilon\in(0,1]}$, and let $(\psi_\epsilon,r_\epsilon)\in C_t([-T,T];H_\epsilon^1(\mathbb{R})\times H_\epsilon^1(\mathbb{R}))$ be the solution  to the system \eqref{system eq}. Then,  for any $\delta>0$, we have
$$\lim_{\epsilon\to 0}\left(\|P_{>\delta\epsilon^{-1/3}}\psi_\epsilon\|_{X_\epsilon^{1,b}([-T,T])}+\|P_{>\delta\epsilon^{-1/3}}(|\psi_\epsilon|^2\psi_\epsilon)\|_{L_t^{2b}([-T,T]; L_x^2(\mathbb{R}))}\right)=0,$$
where $P_{>N}=\sum_{M>\frac{N}{2}}P_M$. 
\end{proposition}

For the proof, one may try to estimate the $\psi_\epsilon$-equation in the system \eqref{system eq} putting the high frequency projection $P_{>\delta\epsilon^{-1/3}}$. However, one would immediately realize that the frequency cut-off does not work properly, because frequency cut-offs are blurred in products. Indeed, the difference between $P_{>\delta\epsilon^{-1/3}}(|\psi_\epsilon|^2\psi_\epsilon)$ and $|P_{>\delta\epsilon^{-1/3}}\psi_\epsilon|^2P_{>\delta\epsilon^{-1/3}}\psi_\epsilon$ is nonzero in general.

To resolve the technical issue, we employ the operator $m_N(D)$ with multiplier
$$m_N(\xi):=\left\{\begin{aligned}
&\frac{|\xi|}{N} &&\textup{if }|\xi|\leq N,\\
&1 &&\textup{if }|\xi|\geq N.
\end{aligned}\right.$$

\begin{remark}\label{m operator properties}
$(i)$ $m_N(D)$ acts like a high frequency cut-off in the sense that $m_N(\xi)\equiv 1$ for high frequencies $|\xi|\geq N$ but it is arbitrarily small in very low frequencies, i.e., $m_N(\xi)\leq \delta$ if $|\xi|\leq\delta N$ for any small $\delta>0$.\\
$(ii)$ $m_N(\xi_1)$ and $m_N(\xi_2)$ are comparable if $|\xi_1|\sim|\xi_2|$, i.e., $m_N(\xi)\sim m_N(2\xi)$ for all $\xi$.  For this reason, $m_N(D)$ has an advantage in handling the blurring effect in a product.\\
$(iii)$ $m_N(D)$ acts like the differential operator $\frac{1}{N}|\partial_x|$ in low frequencies $|\xi|\leq N$.
\end{remark}

\begin{proof}[Proof of Proposition \ref{core high frequency}]
We claim that
\begin{equation}\label{high frequency estimate claim}
\begin{aligned}
\|m_N(D)\psi_\epsilon\|_{X_\epsilon^{1,b}}&\lesssim \|m_N(D)\psi_{\epsilon,0}\|_{H_\epsilon^1},\\
\|m_N(D)(|\psi_\epsilon|^2\psi_\epsilon)\|_{L_t^{2b}L_x^2}&\lesssim \|m_N(D)\psi_{\epsilon,0}\|_{H_\epsilon^1}.
\end{aligned}
\end{equation}
Indeed, applying the Fourier multiplier $m_N(D)$ to the $\psi_\epsilon$-equation in the system \eqref{system eq} and then estimating by Strichartz estimates, we obtain
\begin{equation}\label{high frequency estimate claim proof}
\begin{aligned}
\|m_N(D)\psi_\epsilon\|_{X_\epsilon^{1,b}}&\lesssim \|m_N(D)\psi_{\epsilon,0}\|_{H_\epsilon^1}+\|m_N(D)(|\psi_\epsilon|^2\psi_\epsilon)\|_{X^{0,-(1-b)}}\\
&\lesssim \|m_N(D)\psi_{\epsilon,0}\|_{H_\epsilon^1}+\|m_N(D)(|\psi_\epsilon|^2\psi_\epsilon)\|_{L_t^{2b}L_x^2}.
\end{aligned}
\end{equation}
For the nonlinear term, we decompose 
$$\begin{aligned}
m_N(D)\big(|\psi_\epsilon|^2\psi_\epsilon\big)&=m_N(D)\big(|P_{\leq N}\psi_\epsilon|^2P_{\leq N}\psi_\epsilon\big)+m_N(D)\big(|\psi_\epsilon|^2\psi_\epsilon-|P_{\leq N}\psi_\epsilon|^2P_{\leq N}\psi_\epsilon\big)\\
&=:I+I\!I,
\end{aligned}$$
where $P_{\leq N}=\sum_{M\leq \frac{N}{2}}P_M$. For the former term $I$, we note that $|P_{\leq N}\psi_\epsilon|^2P_{\leq N}\psi_\epsilon$ is localized in $|\xi|\leq 6N$ on the Fourier side, where $m_N(D)$ behaves like $\sim\frac{|D|}{N}$ (see Remark \ref{m operator properties} $(ii)$ and $(iii)$). Thus, the Leibniz rule and the trilinear estimate (Lemma \ref{basic trilinear estimate}) yield
$$\begin{aligned}
\|I\|_{L_t^{2b}L_x^2}&\sim \frac{1}{N}\|\partial_x(|P_{\leq N}\psi_\epsilon|^2P_{\leq N}\psi_\epsilon)\|_{L_t^{2b}L_x^2}\lesssim T^{\frac{1-b}{2b}} \|\psi_\epsilon\|_{X_\epsilon^{1,b}}^2\frac{1}{N}\|\partial_xP_{\leq N}\psi_\epsilon\|_{X^{0,b}}\\
&\lesssim T^{\frac{1-b}{2b}} \|m_N(D)\psi_\epsilon\|_{X^{0,b}}.
\end{aligned}$$
For the latter term $I\!I$, using the trivial bound $m_N(\xi)\leq 1$ and the trilinear estimate (Lemma \ref{basic trilinear estimate}), we show that 
$$\begin{aligned}
\|I\!I\|_{L_t^{2b}L_x^2}&\leq\big\||\psi_\epsilon|^2\psi_\epsilon-|P_{\leq N}\psi_\epsilon|^2P_{\leq N}\psi_\epsilon\big\|_{L_t^{2b}L_x^2}\lesssim T^{\frac{1-b}{2b}} \|\psi_\epsilon\|_{X_\epsilon^{1,b}}^2\|P_{>N}\psi_\epsilon\|_{X^{0,b}}\\
&\sim T^{\frac{1-b}{2b}}\|m_N(D)\psi_\epsilon\|_{X^{0,b}},
\end{aligned}$$
where $m_N(\xi)\sim1$ for $|\xi|\gtrsim N$ is used in the last step. Hence, taking smaller $T>0$ if necessary, we prove that 
$$\|m_N(D)(|\psi_\epsilon|^2\psi_\epsilon)\|_{L_t^{2b}L_x^2}\leq \frac{1}{2}\|m_N(D)\psi_\epsilon\|_{X^{0,b}},$$
from which \eqref{high frequency estimate claim} follows (see \eqref{high frequency estimate claim proof}).

Now, we fix $\delta>0$ and set $N=\delta\epsilon^{-1/3}$. Then, by the claim \eqref{high frequency estimate claim}, it suffices to show that $\|m_{\delta\epsilon^{-1/3}}(D)\psi_{\epsilon,0}\|_{L^2}\to 0$ (see Remark \ref{m operator properties} $(i)$). Indeed, for any $\delta_0\in (0,\delta]$, we have 
$$\begin{aligned}
\|m_{\delta\epsilon^{-1/3}}(D)\psi_{\epsilon,0}\|_{L^2}&\leq \|m_{\delta\epsilon^{-1/3}}(D)P_{\leq\delta_0\epsilon^{-1/3}}\psi_{\epsilon,0}\|_{L^2}+\|m_{\delta\epsilon^{-1/3}}(D)P_{>\delta_0\epsilon^{-1/3}}\psi_{\epsilon,0}\|_{L^2}\\
&\lesssim\frac{\delta_0}{\delta}\|\psi_{\epsilon,0}\|_{L^2}+\|P_{>\delta_0\epsilon^{-1/3}}\psi_{\epsilon,0}\|_{L^2}.
\end{aligned}$$
where we used the properties of $m_N(D)$ in the last step (see Remark \ref{m operator properties}). Consequently, by \textup{\textbf{(H1)}} and \textup{\textbf{(H2)}}, we obtain $\|m_{\delta\epsilon^{-1/3}}(D)\psi_{\epsilon,0}\|_{L^2}\lesssim \frac{\delta_0}{\delta}+o_\epsilon(1)$. However, since $\delta_0>0$ is arbitrarily,  this completes the proof of the proposition.
\end{proof}

\begin{remark}
Repeating the proof of Proposition \ref{core high frequency} to the NLS \eqref{NLS}, one can show that for any $\delta>0$,  
\begin{equation}\label{NLS property 3}
\lim_{\epsilon\to 0}\|P_{>\delta\epsilon^{-1/3}}(|\psi_\epsilon^{\textup{(NLS)}}|^2\psi_\epsilon^{\textup{(NLS)}})\|_{L_t^1([-T,T]; H_\epsilon^1(\mathbb{R}))}=0.
\end{equation}
Indeed, the proof of \eqref{NLS property 3} is easier, since one may employ the simpler Strichartz norm $\|\cdot\|_{C_t([-T,T]; H_\epsilon^1(\mathbb{R}))\cap L_t^4([-T,T]; L_x^\infty(\mathbb{R}))}$ rather than the Fourier restriction norm.
\end{remark}

\section{Remainder estimate}\label{sec: remainder}

In Proposition \ref{LWP for system}, a preliminary bound $\|r_\epsilon\|_{X_\epsilon^{1,b}([-T,T])}\lesssim T^{\frac{1-b}{2b}}R^3$ is obtained for the remainder term. In this section, we show that the remainder vanishes as $\epsilon\to 0$. 

\begin{proposition}[Remainder estimate]\label{remainder estimate}
Let $\epsilon\in(0,1]$. We assume \textbf{\textup{(H1)}} and \textbf{\textup{(H2)}} for $\{\psi_{\epsilon,0}\}_{\epsilon\in(0,1]}$, and let $(\psi_\epsilon,r_\epsilon)\in C_t([-T,T];H_\epsilon^1(\mathbb{R})\times H_\epsilon^1(\mathbb{R}))$ be the solution  to the system \eqref{system eq}. Then, we have 
\begin{equation}\label{remainder estimate; eq1}
\lim_{\epsilon\to 0}\|r_\epsilon(t)\|_{C_t([-T,T];H_\epsilon^1(\mathbb{R}))}=0.
\end{equation}
If we further assume \textbf{\textup{(H2')}} for some $s>0$, then
\begin{equation}\label{remainder estimate; eq2}
\|r_\epsilon(t)\|_{C_t([-T,T];H_\epsilon^1(\mathbb{R}))}\lesssim\epsilon^{\min\{s,1-\eta\}}\quad\textup{for any small }\eta>0.
\end{equation}
\end{proposition}

The proof of the proposition is reduced to that of the following lemma.

\begin{lemma}[Reduction to the core profile estimate]\label{remainder estimate reduction lemma}
Let $\epsilon\in(0,1]$, and let $\eta>0$ be a sufficiently small number which does not depend on $\epsilon\in(0,1]$. We assume that $\{\psi_{\epsilon,0}\}_{\epsilon\in(0,1]}$ satisfies \textup{\textbf{(H1)}}. Then, the solution $(\psi_\epsilon, r_\epsilon)\in X_\epsilon^{1,1-\eta}([-T,T])\times X_\epsilon^{1,1-\eta}([-T,T])$ to the system \eqref{system eq} satisfies
$$\|r_\epsilon\|_{X_\epsilon^{1,\frac{1}{2}+\eta}([-T,T])}\lesssim \epsilon^{1-2\eta}+\|P_{>\frac{1}{100\epsilon}}\psi_\epsilon\|_{X^{0,\frac{1}{2}+\eta}([-T,T])}.$$
\end{lemma}

\begin{proof}[Proof of Proposition \ref{remainder estimate}, assuming Lemma \ref{remainder estimate reduction lemma}]
By Proposition \ref{core high frequency}, Lemma \ref{remainder estimate reduction lemma} yields \eqref{remainder estimate; eq1}. For \eqref{remainder estimate; eq2}, we recall from Proposition \ref{LWP for system} that $\|\psi_\epsilon\|_{X^{s,\frac{1}{2}+\eta}}\lesssim 1$ if \textbf{\textup{(H2')}} holds. Thus, it follows that $\|P_{>\frac{1}{100\epsilon}}\psi_\epsilon\|_{X^{0,\frac{1}{2}+\eta}}\lesssim \epsilon^s$.
\end{proof}

For the proof of the lemma, we employ the following integral estimate. 

\begin{lemma}\label{key integral inequality}
For $\tau,\xi\in\mathbb{R}$, we define
$$\mathcal{I}^\pm(\tau,\xi):=\int_{-\infty}^\infty\int_{-\infty}^\infty \frac{\mathbf{1}_{\{|\xi_1|, |\xi_2|, |\xi-\xi_1-\xi_2|\leq\frac{1}{50\epsilon}\}}d\xi_1d\xi_2}{\langle \tau+\frac{1}{\epsilon^2}(\pm\langle 1+\epsilon\xi_1\rangle+\langle 1+\epsilon\xi_2\rangle+\langle 1+\epsilon(\xi-\xi_1-\xi_2)\rangle)\rangle^{2(1-\eta)}}.$$
Then, for $\epsilon\in(0,1]$ and sufficiently small $\eta>0$, we have 
$$\mathcal{I}^\pm(\tau,\xi)\lesssim \max\left\{\frac{1}{\langle \tau+\frac{2\sqrt{2}\pm\sqrt{2}+\frac{1}{10}}{\epsilon^2}\rangle^{1-2\eta}}, \frac{1}{\langle \tau+\frac{2\sqrt{2}\pm\sqrt{2}- \frac{1}{10}}{\epsilon^2}\rangle^{1-2\eta}}\right\}.$$
\end{lemma}

\begin{proof}
In a sequel, we always assume that $|\xi_1|, |\xi_2|, |\xi-\xi_1-\xi_2|\leq\frac{1}{50\epsilon}$. First, in the integral $\mathcal{I}^\pm(\tau,\xi)$, we will substitute the function in the bracket by 
$$y_2=y_2(\xi_2):=\tau+\frac{1}{\epsilon^2}\Big(\pm\langle 1+\epsilon\xi_1\rangle+\langle 1+\epsilon\xi_2\rangle+\langle 1+\epsilon(\xi-\xi_1-\xi_2)\rangle\Big).$$
Indeed, $y_2(\xi_2)$ is a convex function having its unique minimum at $\xi_2=\frac{\xi-\xi_1}{2}$, because $y_2'(\frac{\xi-\xi_1}{2})=0$ and $y_2''(\xi_2)=\frac{1}{\langle 1+\epsilon\xi_2\rangle^3}+\frac{1}{\langle 1+\epsilon(\xi-\xi_1-\xi_2)\rangle^3}\sim 1$. As a consequence, by the mean value theorem, we have $|\frac{dy_2}{d\xi_2}|=|y_2'(\xi_2)-y_2'(\frac{\xi-\xi_1}{2})|\sim|\xi_2-\frac{\xi-\xi_1}{2}|$, while by Taylor's theorem, $y_2(\xi_2)-y_2(\frac{\xi-\xi_1}{2})=\frac{y_2''(c)}{2}(\xi_2-\frac{\xi-\xi_1}{2})^2\sim (\xi_2-\tfrac{\xi-\xi_1}{2})^2$. Hence, it follows that
$$\left|\frac{dy_2}{d\xi_2}\right|\sim\sqrt{y_2(\xi_2)-y_2(\tfrac{\xi-\xi_1}{2})}.$$
Moreover, we have
$$\Big|y_2-\tau-\frac{1}{\epsilon^2}(2\sqrt{2}\pm\sqrt{2})\Big|\leq\frac{1}{10\epsilon^2} .$$
Therefore, substituting $\xi_2$ by $y_2$ separately on $(-\infty,\frac{\xi-\xi_1}{2}]$ and $[\frac{\xi-\xi_1}{2},\infty)$, and switching the order of integration, we obtain 
$$\mathcal{I}^\pm(\tau,\xi)\lesssim \int_{\tau+\frac{2\sqrt{2}\pm\sqrt{2}-\frac{1}{10}}{\epsilon^2}}^{\tau+\frac{2\sqrt{2}\pm\sqrt{2}+\frac{1}{10}}{\epsilon^2}} \left\{\int_{|\xi_1|\leq\frac{1}{50\epsilon}}\frac{\mathbf{1}_{\{y_2\geq y_2(\frac{\xi-\xi_1}{2})\}}}{\langle y_2\rangle^{2(1-\eta)}\sqrt{y_2-y_2(\frac{\xi-\xi_1}{2})}}d\xi_1\right\}dy_2.$$

Next, in the inner integral $\{\cdots\}$, we will substitute by 
\begin{equation}\label{y_1}
y_1=y_1(\xi_1):=y_2(\tfrac{\xi-\xi_1}{2})=\tau+\frac{1}{\epsilon^2}\Big(\pm\langle 1+\epsilon\xi_1\rangle+2\langle 1+\epsilon(\tfrac{\xi-\xi_1}{2})\rangle\Big),
\end{equation}
whose first and second derivatives are given by
\begin{align}
y_1'(\xi_1)&=\frac{1}{\epsilon}\left(\pm\frac{1+\epsilon\xi_1}{\langle 1+\epsilon\xi_1\rangle}-\frac{1+\epsilon(\tfrac{\xi-\xi_1}{2})}{\langle 1+\epsilon(\tfrac{\xi-\xi_1}{2})\rangle}\right),\label{y_1 derivatives}\\
y_1''(\xi_1)&=\pm\frac{1}{\langle 1+\epsilon\xi_1\rangle^3}+\frac{1}{2\langle 1+\epsilon(\tfrac{\xi-\xi_1}{2})\rangle^3}.\nonumber
\end{align}
For $\mathcal{I}^+(\tau,\xi)$, we note as before that $y_1(\xi_1)$ is convex, $y_1'(\frac{\xi}{3})=0$, $|\frac{dy_1}{d\xi_1}|\sim |\xi_1-\frac{\xi}{3}|$ and $y_1(\xi_1)-y_1(\tfrac{\xi}{3})\sim (\xi_1-\tfrac{\xi}{3})^2$. Thus, the substitution $y_1=y_2(\frac{\xi-\xi_1}{2})$ yields
$$\mathcal{I}^+(\tau,\xi)\lesssim \int_{\tau+\frac{3\sqrt{2}-\frac{1}{10}}{\epsilon^2}}^{\tau+\frac{3\sqrt{2}+\frac{1}{10}}{\epsilon^2}}\left\{\int_{y_1(\frac{\xi}{3})}^{y_2}\frac{dy_1}{\langle y_2\rangle^{2(1-\eta)}\sqrt{y_2-y_1}\sqrt{y_1-y_1(\tfrac{\xi}{3})}}\right\}dy_2.$$
Then, applying the elementary inequality
$$\int_\alpha^\beta\frac{dy_1}{\sqrt{\beta-y_1}\sqrt{y_1-\alpha}}=\int_0^{\beta-\alpha}\frac{dy_1}{\sqrt{(\beta-\alpha)-y_1}\sqrt{y_1}}=\int_0^1\frac{dy_1}{\sqrt{1-y_1}\sqrt{y_1}}\sim 1\quad (\alpha<\beta),$$
we conclude that 
$$\mathcal{I}^+(\tau,\xi)\lesssim  \int_{\tau+\frac{3\sqrt{2}-\frac{1}{10}}{\epsilon^2}}^{\tau+\frac{3\sqrt{2}+\frac{1}{10}}{\epsilon^2}}\frac{dy_2}{\langle y_2\rangle^{2(1-\eta)}}\sim\max_\pm \left\{\frac{1}{\langle \tau+\frac{3\sqrt{2}\pm \frac{1}{10}}{\epsilon^2}\rangle^{1-2\eta}}\right\}.$$
For $\mathcal{I}^-(\tau,\xi)$, we again substitute by $y_1(\xi_1)=y_2(\frac{\xi-\xi_1}{2})$, but we instead use that $|y_1(\xi_1)-\tau-\frac{\sqrt{2}}{\epsilon^2}|\leq\frac{1}{10\epsilon^2}$ (see \eqref{y_1}) and $-\frac{dy_1}{d\xi_1}\sim\frac{1}{\epsilon}$ (see \eqref{y_1 derivatives}). Then, it follows that 
$$\begin{aligned}
\mathcal{I}^-(\tau,\xi)&\lesssim \int_{\tau+\frac{\sqrt{2}-\frac{1}{10}}{\epsilon^2}}^{\tau+\frac{\sqrt{2}+\frac{1}{10}}{\epsilon^2}} \left\{\int_{\tau+\frac{\sqrt{2}-\frac{1}{10}}{\epsilon^2}}^{\tau+\frac{\sqrt{2}+\frac{1}{10}}{\epsilon^2}}\frac{\epsilon \mathbf{1}_{\{y_2\geq y_1\}}}{\langle y_2\rangle^{2(1-\eta)}\sqrt{y_2-y_1}}dy_1\right\}dy_2\\
&\lesssim  \int_{\tau+\frac{\sqrt{2}-\frac{1}{10}}{\epsilon^2}}^{\tau+\frac{\sqrt{2}+\frac{1}{10}}{\epsilon^2}} \frac{dy_2}{\langle y_2\rangle^{2(1-\eta)}}\sim\max_\pm \left\{\frac{1}{\langle \tau+\frac{\sqrt{2}\pm \frac{1}{10}}{\epsilon^2}\rangle^{1-2\eta}}\right\}.
\end{aligned}$$
\end{proof}

\begin{proof}[Proof of Lemma \ref{remainder estimate reduction lemma}]
For convenience, we denote $P_{low}=P_{\leq\frac{1}{100\epsilon}}$ and $P_{high}=P_{>\frac{1}{100\epsilon}}$. For the $r_\epsilon$-equation in the system \eqref{system eq}, after taking out the low frequency piece $\mathcal{R}_\epsilon(P_{low}\psi_\epsilon)$ from $\mathcal{R}_\epsilon(\psi_\epsilon)$, we apply Strichartz estimates. Then, it follows that 
$$\begin{aligned}
\|r_\epsilon\|_{X_\epsilon^{1,\frac{1}{2}+\eta}}&\lesssim\big\||A_\epsilon|^2A_\epsilon-|\psi_\epsilon|^2\psi_\epsilon+\mathcal{R}_\epsilon(A_\epsilon)-\mathcal{R}_\epsilon(P_{low}\psi_\epsilon)\big\|_{X_\epsilon^{0,-(\frac{1}{2}-\eta)}}+\|\mathcal{R}_\epsilon(P_{low}\psi_\epsilon)\|_{X_\epsilon^{0,-(\frac{1}{2}-\eta)}}\\
&=:I+I\!I.
\end{aligned}$$
For $I$, we apply the trilinear estimate (Lemma \ref{basic trilinear estimate}) with the bounds $\|\psi_\epsilon\|_{X_\epsilon^{1,1-\eta}}, \|r_\epsilon\|_{X_\epsilon^{1,1-\eta}}\lesssim1$ (see Proposition \ref{LWP for system}). Then, it follows that 
$$I\lesssim T^{\frac{1-2\eta}{2(1+2\eta)}}\big(\|r_\epsilon\|_{X^{0,\frac{1}{2}+\eta}}+ \|P_{high}\psi_\epsilon\|_{X^{0,\frac{1}{2}+\eta}}\big).$$
Hence, replacing $T>0$ by smaller one if necessary, we obtain
$$\|r_\epsilon\|_{X_\epsilon^{1,\frac{1}{2}+\eta}}\lesssim I\!I+\|P_{high}\psi_\epsilon\|_{X^{0,\frac{1}{2}+\eta}}.$$
It remains to estimate $I\!I$. Indeed, it is obvious that 
$$\begin{aligned}
I\!I&\lesssim \big\|e^{\frac{2i}{\epsilon}(x-\frac{t}{\epsilon\sqrt{2}})}(P_{low}\psi_\epsilon)^3\big\|_{X^{0,-(\frac{1}{2}-\eta)}}+\Big\|e^{\frac{2i}{\epsilon}(x-\frac{t}{\epsilon\sqrt{2}})}\overline{(e^{\frac{2i}{\epsilon}(x-\frac{t}{\epsilon\sqrt{2}})}P_{low} \psi_\epsilon)^2}(P_{low} \psi_\epsilon)\Big\|_{X^{0,-(\frac{1}{2}-\eta)}}\\
&\quad+\Big\|e^{\frac{2i}{\epsilon}(x-\frac{t}{\epsilon\sqrt{2}})}\overline{(e^{\frac{2i}{\epsilon}(x-\frac{t}{\epsilon\sqrt{2}})}P_{low} \psi_\epsilon)^3}\Big\|_{X^{0,-(\frac{1}{2}-\eta)}}\\
&=:I\!I_1+I\!I_2+I\!I_3.
\end{aligned}$$
Thus, it suffices to show that for $j=1,2,3$,
\begin{equation}\label{I_j estimates}
I\!I_j\lesssim \epsilon^{1-2\eta}\|\psi_\epsilon\|_{X^{0,1-\eta}}^3\lesssim_R \epsilon^{1-2\eta}.
\end{equation}
We consider $I\!I_1$ first. Note that by the Plancherel theorem and the Cauchy-Schwarz inequality, the proof of \eqref{I_j estimates} with $j=1$ is reduced to show a uniform bound for the integral\footnote{It is a typical reduction. We refer to \cite[Lemma 3.1]{Tao} for instance.}
$$\begin{aligned}
\frac{1}{\langle\tau-\frac{\sqrt{2}}{\epsilon^2}+p_\epsilon(\xi+\frac{2}{\epsilon})\rangle^{1-2\eta}}\iint_{\mathbb{R}\times\mathbb{R}}&\iint_{\mathbb{R}\times\mathbb{R}}\frac{\mathbf{1}_{\{|\xi_1|, |\xi_2|, |\xi-\xi_1-\xi_2|\leq\frac{1}{50\epsilon}\}}}{\langle \tau-\tau_1-\tau_2+p_\epsilon(\xi-\xi_1-\xi_2)\rangle^{2(1-\eta)}}\\
&\qquad\cdot \frac{1}{\langle \tau_1+p_\epsilon(\xi_1)\rangle^{2(1-\eta)}\langle \tau_2+p_\epsilon(\xi_2)\rangle^{2(1-\eta)}}d\tau_1d\tau_2d\xi_1d\xi_2
\end{aligned}$$
Indeed, eliminating $\tau_1,\tau_2$-integrations by the elementary inequality
$$\int_{\mathbb{R}}\frac{d\tau}{\langle\tau-\alpha\rangle^{2(1-\eta)}\langle\tau-\beta\rangle^{2(1-\eta)}}\lesssim\frac{1}{\langle\alpha-\beta\rangle^{2(1-\eta)}},$$
it is further reduced to show that 
$$\mathcal{I}_1:=\frac{1}{\langle\tau-\frac{\sqrt{2}}{\epsilon^2}+p_\epsilon(\xi+\frac{2}{\epsilon})\rangle^{1-2\eta}}\iint_{\mathbb{R}\times\mathbb{R}} \frac{\mathbf{1}_{\{|\xi_1|, |\xi_2|, |\xi-\xi_1-\xi_2|\leq\frac{1}{50\epsilon}\}}}{\langle \tau+p_\epsilon(\xi_1)+p_\epsilon(\xi_2)+p_\epsilon(\xi-\xi_1-\xi_2))\rangle^{2(1-\eta)}}d\xi_1d\xi_2$$
is bounded uniformly in $\tau,\xi\in\mathbb{R}$. For its proof, we apply Lemma \ref{key integral inequality} with
$$\begin{aligned}
&\tau+p_\epsilon(\xi_1)+p_\epsilon(\xi_2)+p_\epsilon(\xi-\xi_1-\xi_2))\\
&=\tau-\frac{3\sqrt{2}}{\epsilon^2}-\frac{\xi}{\epsilon\sqrt{2}}+\frac{1}{\epsilon^2}\Big(\langle 1+\epsilon\xi_1\rangle+\langle 1+\epsilon\xi_2\rangle+\langle 1+\epsilon(\xi-\xi_1-\xi_2)\rangle\Big).
\end{aligned}$$
Then, it follows that 
$$\begin{aligned}
\mathcal{I}_1&=\frac{\mathbf{1}_{\{|\xi|\leq\frac{3}{50\epsilon}\}}}{\langle\tau-\frac{\sqrt{2}}{\epsilon^2}+p_\epsilon(\xi+\frac{2}{\epsilon})\rangle^{1-2\eta}}\mathcal{I}^+(\tau-\tfrac{3\sqrt{2}}{\epsilon^2}-\tfrac{\xi}{\epsilon\sqrt{2}},\xi)\\
&\lesssim\frac{\mathbf{1}_{\{|\xi|\leq\frac{3}{50\epsilon}\}}}{\langle\tau-\frac{\xi}{\epsilon\sqrt{2}}-\frac{3\sqrt{2}}{\epsilon^2}+\frac{1}{\epsilon^2}\langle 3+\epsilon\xi\rangle\rangle^{1-2\eta}}\max_\pm\left\{\frac{1}{\langle \tau-\frac{\xi}{\epsilon\sqrt{2}}\pm\frac{1}{10\epsilon^2}\rangle^{1-2\eta}}\right\}\lesssim\epsilon^{2(1-2\eta)}.
\end{aligned}$$
Similarly but using that the Fourier transform of $\overline{e^{\frac{2i}{\epsilon}(x-\frac{t}{\epsilon\sqrt{2}})}A_\epsilon}$ is given by $\overline{\widetilde{A_\epsilon}(-\tau+\frac{\sqrt{2}}{\epsilon^2},-\xi-\frac{2}{\epsilon})}$ and $-\tau+\frac{\sqrt{2}}{\epsilon^2}+p_\epsilon(-\xi-\frac{2}{\epsilon})=-\tau+\frac{\sqrt{2}}{\epsilon^2}+\frac{\xi}{\epsilon\sqrt{2}}+\frac{1}{\epsilon^2}\langle 1+\epsilon\xi\rangle$, the proof of \eqref{I_j estimates} for $j=2,3$ can be respectively reduced to show that 
$$\begin{aligned}
&\sup_{\tau\in\mathbb{R},|\xi|\leq\frac{3}{50\epsilon}}\frac{1}{\langle\tau-\frac{\xi}{\epsilon\sqrt{2}}+\frac{\sqrt{2}}{\epsilon^2}+\frac{1}{\epsilon^2}\langle1-\epsilon\xi\rangle\rangle^{1-2\eta}}\mathcal{I}^-\big(-\tau+\tfrac{\xi}{\epsilon\sqrt{2}}+\tfrac{3\sqrt{2}}{\epsilon^2},\xi\big)\lesssim\epsilon^{2(1-2\eta)},\\
&\sup_{\tau\in\mathbb{R},|\xi|\leq\frac{3}{50\epsilon}}\frac{1}{\langle\tau-\frac{\xi}{\epsilon\sqrt{2}}+\frac{\sqrt{2}}{\epsilon^2}+\frac{1}{\epsilon^2}\langle1-\epsilon\xi\rangle\rangle^{1-2\eta}}\mathcal{I}^+\big(-\tau+\tfrac{\xi}{\epsilon\sqrt{2}}+\tfrac{3\sqrt{2}}{\epsilon^2},\xi\big)\lesssim\epsilon^{2(1-2\eta)},
\end{aligned}$$
but these integral inequalities immediately follow from Lemma \ref{key integral inequality}.
\end{proof}

\section{Proof of the main results}\label{sec: proof of the main results}

We are ready to prove Theorem \ref{main theorem} and Theorem \ref{main theorem'}. Indeed, by the remainder estimate (Proposition \ref{remainder estimate}), it is enough to consider the core profile $\psi_\epsilon$ in the system \eqref{system eq}.

Subtracting the NLS
\begin{equation}\label{integral NLS}
\psi_\epsilon^{\textup{(NLS)}}(t)=e^{\frac{it}{4\sqrt{2}}\partial_x^2}\psi_{\epsilon,0}-\frac{3i}{2\sqrt{2}}\int_0^t e^{\frac{i(t-t_1)}{4\sqrt{2}}\partial_x^2}\big(|\psi_\epsilon^{\textup{(NLS)}}|^2\psi_\epsilon^{\textup{(NLS)}}\big)(t_1)dt_1
\end{equation}
from the core profile equation, the difference of the two flows is written as
$$\begin{aligned}
\psi_\epsilon(t)-\psi_\epsilon^{\textup{(NLS)}}(t)&=(S_\epsilon(t)-e^{\frac{it}{4\sqrt{2}}\partial_x^2})\psi_{\epsilon,0}\\
&\quad-\frac{3i}{2}\int_0^t (S_\epsilon(t-t_1)-e^{\frac{i(t-t_1)}{4\sqrt{2}}\partial_x^2})\frac{1}{\langle 1+\epsilon D\rangle}(|\psi_\epsilon|^2\psi_\epsilon)(t_1)dt_1\\
&\quad -\frac{3i}{2}\int_0^t e^{\frac{i(t-t_1)}{4\sqrt{2}}\partial_x^2}\frac{1}{\langle 1+\epsilon D\rangle}\big(|\psi_\epsilon|^2\psi_\epsilon-|\psi_\epsilon^{\textup{(NLS)}}|^2\psi_\epsilon^{\textup{(NLS)}}\big)(t_1)dt_1\\
&\quad-\frac{3i}{2}\int_0^t e^{\frac{i(t-t_1)}{4\sqrt{2}}\partial_x^2}\left(\frac{1}{\langle 1+\epsilon D\rangle}-\frac{1}{\sqrt{2}}\right)\big(|\psi_\epsilon^{\textup{(NLS)}}|^2\psi_\epsilon^{\textup{(NLS)}}\big)(t_1)dt_1.
\end{aligned}$$
For the third term on the right hand side, we note that by the uniform bounds for nonlinear solutions (\eqref{NLS property 1} for $\psi_\epsilon^{\textup{(NLS)}}$ and \eqref{core bound} with the transference principle for $\psi_\epsilon$), 
$$\begin{aligned}
\left\||\psi_\epsilon|^2\psi_\epsilon-|\psi_\epsilon^{\textup{(NLS)}}|^2\psi_\epsilon^{\textup{(NLS)}}\right\|_{L_t^1L_x^2}&\leq T^{1/2}\left(\|\psi_\epsilon\|_{L_t^4L_x^\infty}^2+\|\psi_\epsilon^{\textup{(NLS)}}\|_{L_t^4L_x^\infty}^2\right)\|\psi_\epsilon(t)-\psi_\epsilon^{\textup{(NLS)}}\|_{C_tL_x^2}\\
&\lesssim T^{1/2} \|\psi_\epsilon(t)-\psi_\epsilon^{\textup{(NLS)}}\|_{C_tL_x^2},
\end{aligned}$$
where $T>0$ is sufficiently small but independent of $\epsilon>0$ and the time interval $[-T,T]$ in norms is omitted for convenience. Hence, we have that for $s=0,1$, 
\begin{equation}\label{main theorem proof estimate}
\begin{aligned}
\|\psi_\epsilon(t)-\psi_\epsilon^{\textup{(NLS)}}(t)\|_{C_tH_\epsilon^s}&\lesssim\|(S_\epsilon(t)-e^{\frac{it}{4\sqrt{2}}\partial_x^2})\psi_{\epsilon,0}\|_{C_tH_\epsilon^s}\\
&\quad+\int_0^T \big\|(S_\epsilon(t-t_1)-e^{\frac{i(t-t_1)}{4\sqrt{2}}\partial_x^2})(|\psi_\epsilon|^2\psi_\epsilon)(t_1)\big\|_{C_tL_x^2}dt_1\\
&\quad+\left\|\left(\frac{1}{\langle 1+\epsilon D\rangle}-\frac{1}{\sqrt{2}}\right)\big(|\psi_\epsilon^{\textup{(NLS)}}|^2\psi_\epsilon^{\textup{(NLS)}}\big)\right\|_{L_t^1H_\epsilon^s}.
\end{aligned}
\end{equation}
For Theorem \ref{main theorem}, we take $s=1$ in \eqref{main theorem proof estimate}, and recall that by the high frequency estimates (\textup{\textbf{(H2)}}, Proposition \ref{core high frequency} and \eqref{NLS property 3}), 
$$\|P_{>\delta\epsilon^{-1/3}}\psi_{\epsilon,0}\|_{H_\epsilon^1}, \|P_{>\delta\epsilon^{-1/3}}(|\psi_\epsilon|^2\psi_\epsilon)\|_{L_t^1L_x^2}, \|P_{>\delta\epsilon^{-1/3}}(|\psi_\epsilon^{\textup{(NLS)}}|^2\psi_\epsilon^{\textup{(NLS)}})\|_{L_t^1L_x^2}\to 0$$
for any $\delta>0$. Thus, by Lemma \ref{linear flow convergence}, we conclude that $\|\psi_\epsilon(t)-\psi_\epsilon^{\textup{(NLS)}}(t)\|_{C_tH_\epsilon^1}\to 0$.

For Theorem \ref{main theorem'}, we assume that \textup{\textbf{(H2')}} holds, and take $s=0$ in \eqref{main theorem proof estimate}. Then, by \eqref{high Sobolev norm core nonlinear bound} and \eqref{NLS property 2}, we have that $\|\psi_{\epsilon,0}\|_{H^s}$,  $\||\psi_\epsilon|^2\psi_\epsilon\|_{L_t^1H_x^s}$ and $\||\psi_\epsilon^{\textup{(NLS)}}|^2\psi_\epsilon^{\textup{(NLS)}}\|_{L_t^1H_x^s}$ are uniformly bounded. Thus, following the proof of Lemma \ref{linear flow convergence}, one can show that $\|\psi_\epsilon(t)-\psi_\epsilon^{\textup{(NLS)}}(t)\|_{C_tL_x^2}\lesssim\epsilon^{\frac{\epsilon}{3}}$. Finally, combining with the remainder estimate (Proposition \ref{remainder estimate}), we complete the proof.

\end{document}